\documentclass[12pt]{article}
\usepackage{amsmath,amssymb,amsthm,epsf, graphics}
\usepackage{graphicx}
\usepackage{xcolor}
\usepackage{mathtools}
\usepackage{float}
\usepackage{mathabx}
\usepackage{hyperref}

\theoremstyle{definition}

\newtheorem{theo}{{\bf{Theorem}}}[section]
\newtheorem{prop}[theo]{{\bf Proposition}}
\newtheorem{lem}[theo]{{\bf Lemma}}

\newtheorem{defi}{{\bf Definition}}[section]

\newtheorem*{cordpropPrime}{Proposition \ref{cordprop}$'$}
\newtheorem*{hbpropPrime}{Proposition \ref{hbprop0}$'$}

\allowdisplaybreaks

\newcommand{\FF}{\mathbb{F}}
\newcommand{\RR}{\mathbb{R}}

\title{A bound for plany Kakeya sets in $\FF_q^4$ using the planebrush method}
\author{Izabella \L{}aba\thanks{Department of Mathematics, University of British Columbia, Vancouver, BC, Canada}\and 
Mukul Rai Choudhuri\footnotemark[1]
\and Joshua Zahl\thanks{Chern Institute of Mathematics and LPMC, Nankai University, Tianjin, China}}
\date{June 2025}

\begin{document}

\maketitle

\begin{abstract}
Katz and Zahl \cite{katzzahl} used a planebrush argument to prove that Kakeya sets in $\RR^4$ have Hausdorff dimension at least 3.059. In the special case when the Kakeya set is {\it plany}, their argument gives a better lower bound of $10/3$. We give a nontechnical exposition of the Katz-Zahl argument for plany Kakeya sets in the finite field setting.

\end{abstract}

\section{Introduction}

A {\em Kakeya set} in $\RR^n$, for $n\geq 2$, is a compact set that contains a line segment pointing in every direction. A classic result of Besicovitch \cite{bes} states that Kakeya sets in $\RR^n$ may have $n$-dimensional Lebesgue measure 0; the {\em Kakeya conjecture} states that they must nonetheless have Hausdorff dimension $n$. This was proved for $n=2$ by Davies \cite{dav}, and, very recently, for $n=3$ by Wang and Zahl \cite{wangzahl}. For 
$n\geq 4$, the conjecture is still open, but there are many partial results obtained using a variety of methods. We refer the reader to \cite{katztao} for an overview of the history of the problem.

The finite field Kakeya problem was introduced by Wolff \cite{wolff} for the purpose of simplifying the exposition of his hairbrush argument from \cite{wolff1}. Subsequently, it took on a life of its own, with finite field analogues of Kakeya and restriction problems considered by many authors \cite{BKT, mockTao, rogers, tao4d}. The Kakeya conjecture in finite fields was eventually resolved by Dvir \cite{dvir}. 

The goal of this article is to give a non-technical exposition of the planebrush argument of Katz and Zahl \cite{katzzahl}. In order to do so, we will work in the finite field setting so that we can focus on the geometrical ideas without the technicalities. As noted above, the finite field Kakeya problem has already been solved. The planebrush argument does not offer further progress in that direction, except for some differences between Dvir's assumptions and ours (see Theorem \ref{thm-pb4}). However, we believe that there is value in giving an accessible exposition of the geometric ideas involved, as they may eventually become part of the solution of the Euclidean Kakeya problem in higher dimensions.

The planebrush argument builds on Wolff's hairbrush argument from 1995 \cite{wolff1}, which gave a lower bound of $(n+2)/2$ on the Hausdorff dimension of Kakeya sets in $\RR^n$. Roughly speaking\footnote{
For the purpose of this rough sketch, we ignore the distinction between infinite lines and finite line segments.
}, the argument then involves finding a line in the Kakeya set with many other lines incident to it (the hairbrush configuration), and lower-bounding the size of the Kakeya set by the size of the hairbrush.

In dimensions 4 and higher, one may try to upgrade the hairbrush to a planebrush, an object where the stem (the line incident to many lines) is replaced by a 2-plane intersected by many of the lines of the Kakeya set. This approach does yield some success in four dimensions, improving on the hairbrush lower bound $(4+2)/2=3$. In \cite{katzzahl}, Katz and Zahl used a planebrush argument to prove that Kakeya sets in $\RR^4$ have Hausdorff dimension at least $3.059$; they further extended it to prove a stronger statement known as the {\em Kakeya maximal function conjecture} for the same range of dimensions. 

Following an idea first introduced by Bourgain and Guth \cite{bourgainguth} in the context of the restriction conjecture, the Katz-Zahl proof is based on considering two extremal cases: the {\em plany case} and the {\em trilinear case}. In the plany case, a typical point $x$ of the Kakeya set has the property that most of the lines passing through $x$ are contained in a common 2-plane (depending on $x$); in the trilinear case, the majority of intersecting triples of lines point in linearly independent directions. The planebrush method works well in the plany case, yielding a bound of $10/3$. In the trilinear case, we have a trilinear Kakeya bound of $3.25$, due to Guth and Zahl \cite{guthzahl}. Interpolating between these two results, Katz and Zahl obtained a bound of $3.059$ for the general case. Note that this is less that $\min(10/3,3.25)$; the additional losses come from having to consider the intermediate cases that are neither fully plany nor fully trilinear. For a special class of Kakeya sets in $\mathbb{R}^4$ called ``sticky Kakeya sets,'' Rai Choudhuri \cite{RaiCh} used a more efficient interpolation to obtain the bound $3.25.$

In this paper, we prove the finite field equivalent of the $10/3$ dimension bound for plany Kakeya sets. The proof is based on the planebrush method, giving us an opportunity to provide an exposition of the corresponding result in the Euclidean case. To that end, we will not assume Dvir's result or use the polynomial method, focusing instead on purely geometric arguments that mimic the Euclidean situation.

\section{The main results}

Let $\FF_q$ be the finite field of size $q$. 
A \emph{line} in $\FF_q^n$ is a one-dimensional affine subspace of $\FF_q^n$. Any line may be written as $\{a+\lambda v:\lambda\in \FF_q\}$, where $a\in \FF_q^n$ is a point on the line and $v\in\FF_q^n\setminus\{0\}$ is its direction vector. Replacing $v$ by any nonzero multiple of $v$ does not change the line; thus we may identify the
set of directions in $\FF_q^n$ with $\mathbb{P}\FF_q^n$, the projective space in $\FF_q^n$. It is easy to see that a line has $q$ points, and that any two lines are disjoint, identical, or intersect in exactly one point. 

\medskip
\begin{defi}
Let $n\geq 2$. A \textit{Kakeya set in }$\FF_q^n$ is a set $K\subset \FF_q^n$ which contains a line in every direction in $\mathbb{P}\FF_q^n$.
\end{defi}

In the Euclidean setting, we are interested in estimates of the form $\dim(K)\geq \alpha$ for $0<\alpha\leq n$, where $K$ is a Kakeya set and $\dim$ stands for some variant of fractal dimension (such as Hausdorff dimension or Minkowski dimension). The analogous finite field estimate is 
\begin{equation}\label{kak-fifi}
|K|\geq C q^\alpha,
\end{equation}
where $C>0$ is a constant that may depend both on $\alpha$ and the ambient dimension $n$, but not on $q$. (For a fixed $q$ and any nonempty $K$, (\ref{kak-fifi}) always holds trivially if we only choose $C$ to be small enough. The point is to make (\ref{kak-fifi}) uniform in $q$.)

Dvir \cite{dvir} established \eqref{kak-fifi} for $\alpha=n$ and thus proved that, in this sense, Kakeya sets in $\FF_q^n$ have dimension $n$.

\begin{theo}\label{ffk}
{\bf (Finite field Kakeya conjecture.)} \cite{dvir}
Let $n\geq 2$. If $K\subset \FF_q^n$ is a Kakeya set, then $|K|\geq C_nq^n$. 
\end{theo}

Theorem \ref{ffk} was proved using the polynomial method, also introduced by Dvir in this context. Further improvements to the constant $C_n$ in Theorem \ref{ffk} were obtained in \cite{BukTi, DKSS}. This completely resolves the finite field version of the Kakeya conjecture.

In this article, we will not rely on Theorem \ref{ffk} or, more generally, on the polynomial method. Instead, we will follow a chain of geometric arguments borrowed from the Euclidean case and culminating in the planebrush estimate. 

Throughout the rest of this paper, we assume that $n\geq 2$ and that $q$ is a large enough prime ($q>600$ will suffice). All constants in the inequalities below may depend on $n$, but not on $q$. We use $A\gtrsim B$ or $B = O(A)$ to say that $A\geq CB$, where $C>0$ is a constant that may change from line to line, but is always independent of $q$. We also write $A\eqsim B$ to say that both $A\gtrsim B$ and $B\gtrsim A$ hold.

Before stating our main result, we recall the following elementary estimate based on an argument of Cord\'oba \cite{cordoba}. The argument does not use any geometrical information about lines except for the fact that two distinct lines in $\FF^n_q$ intersect in at most one point. This estimate is sufficient to resolve the Kakeya problem in 2 dimensions, but does not offer further improvements in higher dimensions. 

\begin{prop}\label{cordprop}
    Let $\mathcal{L}$ be a set of distinct lines in $\FF_q^n$ such that $|\mathcal{L}|\leq 2q$. Then
    \[
        \Big|\bigcup_{l\in\mathcal{L}}l\Big| \gtrsim |\mathcal{L}| q. 
    \]
    In particular, if $K$ is a Kakeya set in $\FF_q^n$, $n\geq 2,$ then we have $|K|\gtrsim q^2$. 

\end{prop}

The last conclusion follows since every Kakeya set in $\FF_q^n$, $n\geq 2$, contains a set of $q+1$ distinct lines. 

Next, we recall an estimate that follows from Wolff's hairbrush argument \cite{wolff1}. This gives a stronger result for Kakeya sets in $\mathbb{F}_q^n$ when $n\geq 3$. 

\begin{prop}\label{hbprop0}
Let $\mathcal{L}$ be a set of distinct lines in $\FF_q^n$ such that at most $2q$ lines in $\mathcal{L}$ are contained in a common 2-plane (this will hold, for example, if the lines point in different directions) and $|\mathcal{L}|\leq 3q^2$. Then
    \[
        \Big|\bigcup_{l\in\mathcal{L}}l\Big| \gtrsim |\mathcal{L}| q^{1/2}.
    \]
In particular, if $K$ is a Kakeya set in $\FF_q^3$, then $|K|\gtrsim q^{5/2}$. 
\end{prop}

To prove Proposition \ref{hbprop0}, we foliate $\FF^n_q$ into 2-planes and apply Proposition \ref{cordprop} in each plane. The statement of Proposition \ref{hbprop0} is tailored to our intended application in 4 dimensions, but an almost identical argument shows that Kakeya sets in $\FF_q^n$ have size $\gtrsim q^{(n+2)/2}$ (see \cite{wolff} for the details).

Following naively the pattern in Propositions \ref{cordprop} and \ref{hbprop0}, one might expect further improvements for families of lines $\mathcal{L}$ in dimensions $n\geq 4$ such that, for any $2\leq d\leq n$, at most $C_d q^{d-1}$ lines of $\mathcal{L}$ are contained in a common $d$-dimensional affine subspace of $\FF_q^n$. Such assumptions are known in the literature as {\em Wolff axioms}. However, this turns out to be already false for $n=4$. The example in \cite[Proposition 1.3]{tao4d} yields a family of lines in $\FF_q^4$ such that $|\mathcal{L}| \eqsim q^3$ and $\mathcal{L}$ obeys the Wolff axioms, but $|\bigcup_{l\in \mathcal{L}}l|\eqsim q^3$. To prove better bounds in 4 dimensions, we must either use the property that all lines point in different directions, or else we must make additional assumptions on how the lines are distributed. We make such an assumption now.

\begin{defi}
A set of lines $\mathcal{L}$ is said to be \emph{plany} if for every $x\in \bigcup_{l\in \mathcal{L}} l$, there exists a 2-plane $\Pi_x$ such that all the lines from $\mathcal{L}$ passing through $x$ are contained in $\Pi_x$. Furthermore, a Kakeya set $K=\bigcup_{l\in \mathcal{L}} l$ is said to be plany if its associated set of lines is plany.
\end{defi}

Planiness is a property of Kakeya sets that has been studied in the literature \cite{guth2016, KLT}. Our goal is to prove the following theorem, based on a planebrush argument.

\begin{theo}\label{thm-pb4}
Let $\mathcal{L}$ be a plany set of distinct lines in $\FF_q^n$ such that at most $2q$ of the lines are contained in a common 2-plane, at most $3q^2$ of the lines are contained in a common 3-plane, and $|\mathcal{L}|\leq 4q^3$. In particular, this holds if all lines point in different directions. Then 
\begin{equation}\label{thm-pb4MainInequality}
    \Big| \bigcup_{l\in \mathcal{L}}l\Big|\gtrsim |\mathcal{L}|q^{1/3}.
\end{equation}
Therefore, if $K\subset \FF_q^4$ is a plany Kakeya set, then $|K|\gtrsim q^{10/3}$.
\end{theo}


\section{Cord\'oba's argument}

We prove a slightly more general and technical variant of Proposition \ref{cordprop} which will be useful in our arguments below. The proof is similar to that of \cite[Proposition 3.1]{guthbook}.

\begin{lem}\label{cordlemma}
    Let $A_1, A_2,\ldots A_N$ be sets with the property that $|A_i\cap A_j|\leq 1$ for any pair $i\neq j$. Then
        \begin{equation}\label{e-Ai}
        \Big|\bigcup_{i=1}^N A_i\Big|\geq \sum_{i=1}^N |A_i| -\frac{N(N-1)}{2}.
    \end{equation}
    In particular, if $\min_i|A_i|\geq CN$ for some $C\geq 1$, then 
    \begin{equation}\label{consequenceOfE-Ai}
    \Big|\bigcup_{i=1}^N A_i\Big|\geq \big(1-(2C)^{-1}\big) \sum_{i=1}^N |A_i|.
    \end{equation}
\end{lem}

\begin{proof}
    We prove (\ref{e-Ai}) by induction in $N$. For $N=1$, the statement is trivial. Assuming that $N\geq 2$ and (\ref{e-Ai}) holds for $N-1$ sets, we write
    \begin{align*}
        \Big|\bigcup_{i=1}^N A_i\Big|&\geq \Big|\bigcup_{i=1}^{N-1} A_i\Big|        
        + |A_N| - \sum_{i=1}^{N-1} |A_i\cap A_N|
       \\
       &\geq \sum_{i=1}^{N-1} |A_i| -\frac{(N-1)(N-2)}{2}      + |A_N| - (N-1)
       \\
       & \geq \sum_{i=1}^N |A_i| -\frac{N(N-1)}{2},     
    \end{align*}
as claimed. To obtain \eqref{consequenceOfE-Ai}, we compute
$$
\frac{N(N-1)}{2}\leq \frac{N^2}{2}\leq \frac{\sum |A_i|}{2C}.\qedhere
$$

\end{proof}

We will use Lemma \ref{cordlemma} to prove the following mild generalization of Proposition \ref{cordprop}
\begin{cordpropPrime}
    Let $\mathcal{L}$ be a set of distinct lines in $\FF_q^n$ such that $|\mathcal{L}|\leq 2q$. Let $X\subset\FF_q^n$ be a set such that for for every $l\in \mathcal{L}$ we have $|l\cap X|\geq q/300$. Then $|X|\gtrsim |\mathcal{L}| q$.
\end{cordpropPrime}
\begin{proof}
Let $N=\min(|\mathcal{L}|, \lfloor q/300\rfloor )$, and let $l_1,\dots,l_N\in \mathcal{L}$ be a distinct lines. Apply Lemma \ref{cordlemma} to the sets $A_i = l_i\cap X$. The result follows from \eqref{consequenceOfE-Ai} with $C=1$. 
\end{proof}

\section{Wolff's hairbrush argument}\label{hbsec}

To prove Theorem \ref{thm-pb4} we will split the ambient space $\FF_q^n$ into an almost disjoint union of affine 3-planes, and then apply a variant of Wolff's hairbrush argument (Proposition \ref{hbprop0}) inside each 3-plane. To execute this strategy, we will need the following mild generalization of Proposition \ref{hbprop0}.

\begin{hbpropPrime}
Let $\mathcal{L}$ be a set of distinct lines in $\FF_q^n$ such that at most $2q$ lines in $\mathcal{L}$ are contained in a common 2-plane (this will hold, for example, if the lines point in different directions) and $|\mathcal{L}|\leq 3q^2$. Let $X\subset \bigcup_{l\in \mathcal{L}}l$ be a set with the property that for each $l\in\mathcal{L}$, we have $|l\cap X'|\geq q/200$. Then
\begin{equation}\label{lowerBdOnXhbpropPrime}
|X|\gtrsim |\mathcal{L}| q^{1/2}.
\end{equation}
\end{hbpropPrime}
\begin{proof}
    We will find a line (which we will call the stem) that has many other lines incident to it, i.e., a configuration that looks like a hairbrush (See Figure \ref{hairbrush}). 

    \begin{figure}[h]
        \centering
        \includegraphics[scale=0.4]{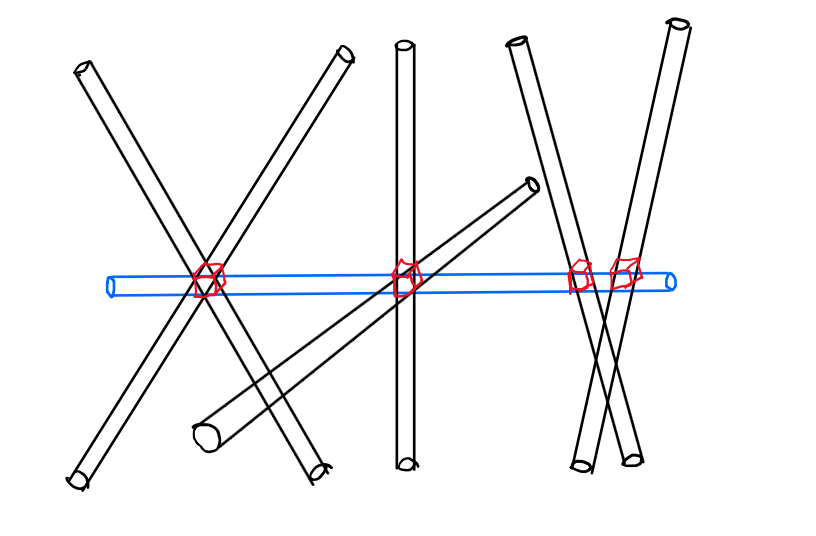}
        \caption[Hairbrush argument]{Hairbrush configuration. We use tubes instead of lines here for ease of illustration (note that working with tubes in the Euclidean setting introduces additional difficulties since tubes might intersect at small angles). Many lines are incident to a single ``stem'' line (shown in blue).}
        \label{hairbrush}
    \end{figure}

    Finding an appropriate stem requires some technical work, which we get to now. For $p\in X$, define the multiplicity function $\mu(p)=\sum_{l \in \mathcal{L}} \chi_{l}(p)$, i.e.,~$\mu(p)$ is the number of lines from $\mathcal{L}$ containing the point $p$. Define
    \[
    T=\{(l,l',p)\in \mathcal{L}\times \mathcal{L}\times X: p\in l \text{ and }p\in l'\}.
    \]
    We have
    \begin{align*}
    |T|= & \sum_{p\in X} \mu(p)^2 \geq \frac{\big(\sum_{p\in X} \mu(p)\big)^2}{|X|}\geq \frac{\big(\sum_{p\in X} \sum_{l \in \mathcal{L}} \chi_l(p)\big)^2}{|X|}\\
    & =  \frac{\sum_{l \in \mathcal{L}}\big(\sum_{p\in X}\chi_l(p)\big)^2}{|X|} 
    =\frac{\big(\sum_{l\in \mathcal{L}}|l\cap X|\big)^2}{|X|} \geq \frac{1}{200^2} \frac{q^2 |\mathcal{L}|^2}{|X|},
    \end{align*}
    where the first inequality used Cauchy-Schwarz. By pigeonholing, there exists a line $l_i$ that occurs as the first coordinate in at least $|T|/|\mathcal{L}|\geq (1/200^2)q^2|\mathcal{L}|/|X|$ many triples. We will have at most $q$ triples of the form $(l_i,l_i,p)$. The remaining triples will be of the form $(l_i,l_j,p)$ with $l_i\neq l_j$. Each triple of this form corresponds to a line incident to $l_i$. Therefore, the total number of lines incident to $l_i$ is at least $(1/200^2)q^2|\mathcal{L}|/|X|-q$. We have two cases now, depending on which one of these two terms is dominant.
    
    \medskip
    
    \underline{Case 1:} Suppose $q\geq \frac{1}{2} \frac{1}{200^2} \frac{q^2 |\mathcal{L}|}{|X|}$. Rearranging, we obtain $|X|\gtrsim |\mathcal{L}| q$ and hence \eqref{lowerBdOnXhbpropPrime}.
    
    \medskip

    \underline{Case 2:} Suppose $q\leq \frac{1}{2} \frac{1}{200^2} \frac{q^2 |\mathcal{L}|}{|X|}$. Then the number of lines incident to $l_i$ is at least $\frac{1}{2} \frac{1}{200^2} \frac{q^2 |\mathcal{L}|}{|X|}$. Thus, $l_i$ is the high multiplicity line that we may use as the stem of the hairbrush. Now we obtain a lower bound on the volume of the hairbrush with stem $l_i$. This will involve dividing $\FF_q^n$ into 2-planes that contain the line $l_i$ and analyzing the lines in each of these 2-planes. 
    
    More precisely, we have the foliation 
    \begin{equation}\label{f3foli}
        \FF_q^3=l_i\sqcup \Big(\bigsqcup_\pi (\pi \setminus l_i) \Big),
    \end{equation} 
    where the union is taken over the set $\pi$ of 2-planes containing $l_i$. Observe that if $l$ is a line incident to $l_i$, then $l$ is contained in exactly one plane from \eqref{f3foli}. We denote the number of lines of the hairbrush that are in $\pi$ as $H(\pi)$. We thus have
    \begin{equation}\label{hbeq2}
    \sum_\pi H(\pi)\geq \frac{1}{2} \frac{1}{200^2} \frac{q^2 |\mathcal{L}|}{|X|}.
    \end{equation}
    We will now establish a lower bound on the union of the lines (intersected with $X$) that intersect $l_i$ and are contained in the plane $\pi$. If $l_j\neq l_i$ is a line of the hairbrush and is contained in the plane $\pi$, then $|l_j\cap X\setminus l_i|\geq q/300$. Furthermore, if $l_j\neq l_j'$ are two distinct lines of this form, then the sets $l_j\cap X\setminus l_i$ and $l_j'\cap X\setminus l_i$ intersect in at most one point. We have $H(\pi)\leq 2q$ by hypothesis. Thus, we may apply Proposition \ref{cordprop}$'$ to these sets. We conclude that
    \begin{equation}\label{XinsidePi}
    \Big| \bigcup_{\substack{l\in \mathcal{L}\\ l\subset\pi}} (X\cap l) \backslash l_i\Big| \gtrsim H(\pi) q.
    \end{equation}
    The sets $\pi\setminus l_i$ are disjoint as per our foliation \eqref{f3foli}. Thus by \eqref{hbeq2} and \eqref{XinsidePi}, we have
    $$
    |X|\gtrsim \sum_\pi H(\pi) q \gtrsim \frac{q^3 |\mathcal{L}|}{|X|}.
    $$
    Simplifying, and using the hypothesis $|\mathcal{L}|\lesssim q^2$, we get
    $$
    |X|\gtrsim q^{3/2} |\mathcal{L}|^{1/2}\gtrsim q^{1/2} |\mathcal{L}|.\qedhere
    $$        
\end{proof}

\section {The planebrush argument}
\begin{proof}[Proof of Theorem \ref{thm-pb4}]
We will prove the result by strong induction on $|\mathcal{L}|$. The base case is trivial, so we assume that $|\mathcal{L}|>1$ and that the theorem holds for smaller values of $|\mathcal{L}|$. Define $X = \bigcup_{l\in\mathcal{L}}l$. For each $p\in X$ we define the multiplicity $\mu (p)=\sum_{l\in \mathcal{L}} \chi_l (p)$, that is, the number of lines from $\mathcal{L}$ containing $p$.
Note that the average multiplicity of a point in $X$ is 
$$
\frac{\sum_{p\in X} \mu(p)}{|X|}=\frac{\sum_{p\in X} \sum_{l\in \mathcal{L}} \chi_l (p)}{|X|}=\frac{\sum_{l\in \mathcal{L}} \sum_{p\in X}  \chi_l (p)}{|X|}=\frac{q|\mathcal{L}|}{|X|}.
$$ 
In the arguments that follow, it will be helpful if each point in $X$ has roughly average multiplicity. Define
$$
X'=\bigg\{ p\in X: \mu (p) \geq \frac{1}{100}\frac{q|\mathcal{L}|}{|X|}\bigg\}.
$$
We expect $X'$ to be a large subset of $X$ since we are only discarding those points that have a small fraction of the average multiplicity. To be more precise, we claim the following.

\medskip

\underline{Claim 1 (Multiplicity):} At least 99\% of the point-line pairs are retained when we shrink from the full Kakeya set $X$ to the subset $X'$. That is,

\begin{equation}\label{1}
    \frac{1}{|\mathcal{L}|}\sum_{l\in \mathcal{L}} \frac{|l \cap X'|}{q}\geq \frac{99}{100}.
\end{equation}

\medskip

We shall now prove this claim. Consider the sum
\begin{equation}\label{sumPinX}
\sum_{p\in X\setminus X'} \mu(p)=\sum_{p\in X\setminus X'} \sum_{l\in \mathcal{L}} \chi_l (p)=\sum_{l\in \mathcal{L}} \sum_{p\in X\setminus X'}  \chi_l (p)=\sum_{l\in \mathcal{L}} |l \cap (X\setminus X')|.
\end{equation}
On the other hand, we have 
\begin{equation}\label{pinXMinusXPrime}
\sum_{p\in X\setminus X'} \mu(p) \leq \sum_{p\in X\setminus X'} \frac{1}{100}\frac{q|\mathcal{L}|}{|X|}\leq \frac{1}{100}\frac{q|\mathcal{L}|}{|X|} |X\setminus X'|\leq \frac{q|\mathcal{L}|}{100}.
\end{equation}
Comparing \eqref{sumPinX} and \eqref{pinXMinusXPrime}, we have
$$
\frac{1}{|\mathcal{L}|}\sum_{l\in \mathcal{L}} \frac{|l \cap (X\setminus X')|}{q}\leq \frac{1}{100}.
$$
This says that, on average, each line loses less than $1\%$ of its points as we restrict from $X$ to $X'$. Since $|l \cap X'|=q-|l \cap (X\setminus X')|$, we can substitute and rearrange to get
\begin{equation*}
    \frac{1}{|\mathcal{L}|}\sum_{l\in \mathcal{L}} \frac{|l \cap X'|}{q}\geq \frac{99}{100},
\end{equation*}
which establishes Claim 1.

\medskip

The next step is to find a point that serves as the base of our planebrush. 

\medskip

\underline{Claim 2 (Planebrush base):} There exists a point $x_1\in X'$ such that at least half of the lines passing through $x_1$ still have at least half of their points in $X'$, i.e.
\begin{equation}\label{atLeastHalfLines}
|\{l\in\mathcal{L}\colon x_1\in l, |X'\cap l|\geq q/2\}|\geq \frac{1}{2} |\{l\in\mathcal{L}\colon x_1\in l\}|.
\end{equation}

\medskip

We establish this claim as follows. Suppose for contradiction that \eqref{atLeastHalfLines} fails for every $x\in X'$. Let us count the number of triples in the set $T=\{(x,p,l): x,p\in X', l\in \mathcal{L}, x,p\in l\}$. We have
\begin{align*}
    |T|&\leq \sum_{x\in X'} \Big[ \sum_{\substack{l\ni x\\ |l\cap X'|< q/2}}\frac{q}{2} + \sum_{\substack{l\ni x\\|l\cap X'|\geq q/2}} q\Big]\\
       &=\sum_{x\in X'} \Big[ \Big(|\{l:x\in l,|l\cap X'|< q/2\}|\Big)\frac{q}{2}+\Big(\mu(x)-|\{l:x\in l,|l\cap X'|< q/2\}|\Big)q\Big]\\
       &=\sum_{x\in X'} \Big[ \Big(\mu(x)-\frac{1}{2}|\{l:x\in l,|l\cap X'|< q/2\}|\Big)q\Big]\\
       &\leq \sum_{x\in X'} \Big[ \Big( \mu(x)-\frac{\mu(x)}{2\cdot 2}\Big)q\Big]\\
       &=\frac{3}{4} q \sum_{x\in X'} \mu(x)\\
       &\leq \frac{3}{4}q^2 |\mathcal{L}|.
\end{align*}
On the other hand, we have $|T|=\sum_{l\in \mathcal{L}} |l \cap X'|^2$. By Cauchy-Schwarz inequality and \eqref{1}, we get
$$
|T|=\sum_{l\in \mathcal{L}} |l \cap X'|^2\geq \frac{\big(\sum_{l\in \mathcal{L}} |l \cap X'|\big)^2}{|\mathcal{L}|}\geq \bigg(\frac{99}{100}\bigg)^2 q^2 |\mathcal{L}|.
$$
Hence we arrive at the required contradiction and this establishes Claim 2. Henceforth we will fix the base point $x_1$ of our planebrush. We will also refer to the set $\bigcup_{l\in \mathcal{L}, x_1\in l} l$ as the \emph{bush} centered at $x_1$. By our planiness assumption, this set is contained in a 2-plane $\Pi_{x_1}$.

\begin{figure}[h]
    \centering
        \centering
        \includegraphics[scale=0.4]{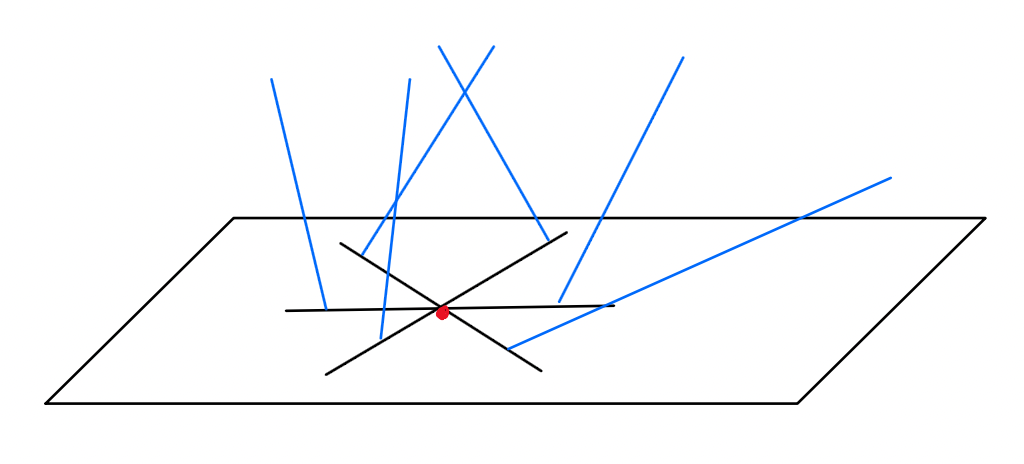} 
        
        \label{pbfig}
    \caption[Planebrush configuration]{Planebrush configuration. The red dot represents the base point $x_1$ and the blue lines represent some of the lines of the planebrush $\mathcal{L}_1$.}

\end{figure} 

Next we define $\mathcal{L}_1$ to be the set of lines $l\in \mathcal{L}$ which either intersect $\Pi_{x_1}$ or are parallel to $\Pi_{x_1}$ (the latter lines intersect $\Pi_{x_1}$ at infinity). This forms our planebrush configuration (see Figure \ref{pbfig}). Since $x_1$ satisfies \eqref{atLeastHalfLines}, we have
\begin{equation}\label{pbcount}
    |\mathcal{L}_1|\geq  10^{-5}\ \frac{q^3|\mathcal{L}|^2}{|X|^2}-2q^2.
\end{equation}

Let us elaborate on Inequality \eqref{pbcount}. 
Let $p_1, p_2,\ldots, p_N$ be the set of points of the bush centered at $x_1$ (excluding $x_1$ itself) that are in $X'$.  Since $x_1$ satisfies \eqref{atLeastHalfLines}, we have $N\geq 10^{-3}\ \frac{q|\mathcal{L}|}{|X|}q$. Let $\mathcal{L}_{p_i}\subset\mathcal{L}$ be the set of lines passing through $p_i$. We have $|\mathcal{L}_{p_i}|\geq 10^{-2}\ \frac{q|\mathcal{L}|}{|X|}$, and the first term in \eqref{pbcount} is a lower bound on the sum $\sum_{i=1}^N |\mathcal{L}_{p_1}|$. However, these sets $\mathcal{L}_{p_i}$ might not be disjoint. The second term fixes the potential double counting. Suppose $l\in \mathcal{L}_{p_i}\cap \mathcal{L}_{p_j}$. Since $p_i,p_j\in l$, we have that $l$ is contained in the plane $\Pi_{x_1}$ (see Figure \ref{doubcount} below). 

\begin{figure}[h]
    \centering
        \centering
        \includegraphics[scale=0.4]{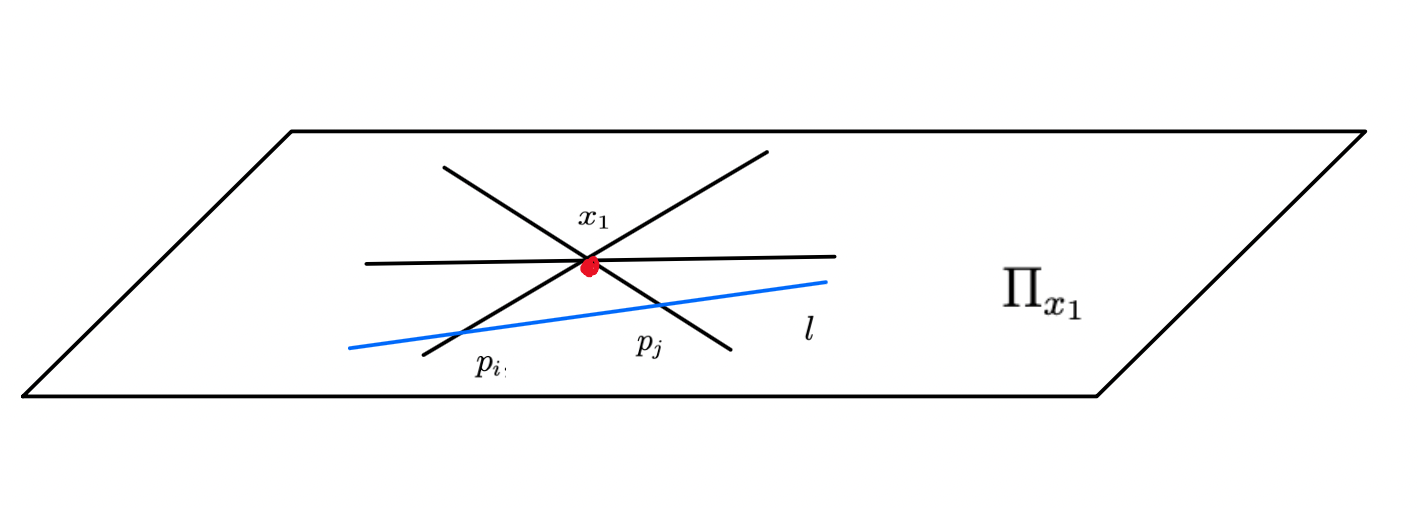} 
        \caption{Double counting. The blue line $l$ is being counted twice as it passes through both $p_i$ and $p_j$. This forces $l$ to be contained in $\Pi_{x_1}$.}
        \label{doubcount}
\end{figure} 

Let $\mathcal{L}_0\subset\mathcal{L}_1$ be the set of lines contained in the plane $\Pi_{x_1}$. By hypothesis $|\mathcal{L}_0|\leq 2q$. These are the only lines which might be counted more than once, and since each line contained $q$ points, each of these lines is counted at most $q$ times. Thus, we subtract $2q^2$ to remove lines from $\mathcal{L}_0$ from the count.  This explains the second term of Inequality \eqref{pbcount}. 

Suppose first that $2q^2\geq 10^{-5}\ \frac{q^3|\mathcal{L}|^2}{|X|^2}$. Then upon simplification, we get $|X|\gtrsim |\mathcal{L}| q^{1/2}$ and \eqref{thm-pb4MainInequality} follows. Henceforth we shall suppose that $2q^2\leq 10^{-5}\ \frac{q^3|\mathcal{L}|^2}{|X|^2}$, and thus
\begin{equation}\label{pbcard}
    |\mathcal{L}_1|\gtrsim \frac{q^3|\mathcal{L}|^2}{|X|^2}.
\end{equation}

Define the planebrush $P_1=\bigcup_{l\in \mathcal{L}_1} l$. Our next task is to obtain a lower bound for $|P_1|$. To do this, define $S=\{(p,l): p\in l \text{ and } l\in \mathcal{L}_1\}$ to be the set of point-line incidence pairs in the planebrush $P_1$. We split into two cases.

\medskip

\underline{Case 1 (Planebrush lines are essentially disjoint):} For at least half the pairs $(p,l)$ in $S$, we have that $l$ is the unique line from $\mathcal{L}_1$ containing $p$. Roughly speaking, this corresponds to the situation when the lines of the planebrush are disjoint. Let us find a lower bound on the volume of $|P_1|$ in this case. Let $S'$ be the set of pairs $(p,l)$ in $S$ such that $l$ is the unique line from $\mathcal{L}_1$ containing $p$. The projection to the first coordinate of the pairs in $S'$ is injective by definition. Therefore, we get

$$
|X|\geq |P_1|\geq |S'|\geq \frac{|S|}{2}=\frac{q|\mathcal{L}_1|}{2}\gtrsim \frac{q^4|\mathcal{L}|^2}{|X|^2}.  
$$
We used \eqref{pbcard} for the final inequality above. Rearranging, we get 
$$
|X|\gtrsim q^{4/3}|\mathcal{L}|^{2/3}\gtrsim q^{1/3}|\mathcal{L}|,
$$ 
as desired. Note that we used the hypothesis that $|\mathcal{L}|\lesssim q^3$, i.e., $q\gtrsim |\mathcal{L}|^{1/3}$ in the second inequality above. This completes the proof in this case.

\medskip

\underline{Case 2 (Planebrush lines have many intersections):} For at least half the pairs $(p,l)$ in $S$, we have that $l$ is not the unique line from $\mathcal{L}_1$ passing through $p$. Denote the set of such pairs as $S''=S\setminus S'$, with $S'$ defined as in Case 1. Furthermore, let $P_1'\subset P_1$ be the projection of $S''$ to the first coordinate (this will not be injective). That is, $P_1'$ is set of points in $P_1$ that have at least two lines from $\mathcal{L}_1$ passing through them. Our goal is to find a lower bound on $|P_1'|$. We now make the following claim, which we will prove by geometric arguments.

\medskip

\underline{Claim 3 (Separating the planebrush):} The set $P_1'$ is disjoint from the set $\bigcup_{l\in \mathcal{L}_1^c} l$. Therefore, we have
\begin{equation}\label{disjunipb}
    |X|\geq |P_1'|+\bigg|\bigcup_{l\in \mathcal{L}_1^c} l\bigg|.
\end{equation}

\medskip

The utility of this claim is that we can analyze the planebrush in isolation, and then apply the inductive hypothesis to the remaining lines. To prove the claim, let  $p\in P_1'$. By definition, there exists $l_1,l_2\in \mathcal{L}_1$ such that $p\in l_1$ and $p\in l_2$. By planiness, there exists a plane $\Pi_p$ containing all the lines passing through $p$ (even if, a priori, these lines are not part of the planebrush). Moreover, $\Pi_p$ must be the span of the lines $l_1$ and $l_2$. Note that $l_1$ and $l_2$ meet $\Pi_{x_1}$ at two points (these could be at infinity). Hence $\Pi_p$ and $\Pi_{x_1}$ intersect in a line $l$ passing through these two points (See Figure \ref{pbstep}). This is a non-trivial observation since in four and higher dimensions, two generic planes meet in a point. Thus $\Pi_{x_1}$ and $\Pi_p$ are contained inside a 3-space of $\FF_q^n$ (generically you would need a 4-space, i.e. the span of $\Pi_{x_1}$ and $\Pi_p$). 

 \begin{figure}[h]
\includegraphics[scale=0.5]{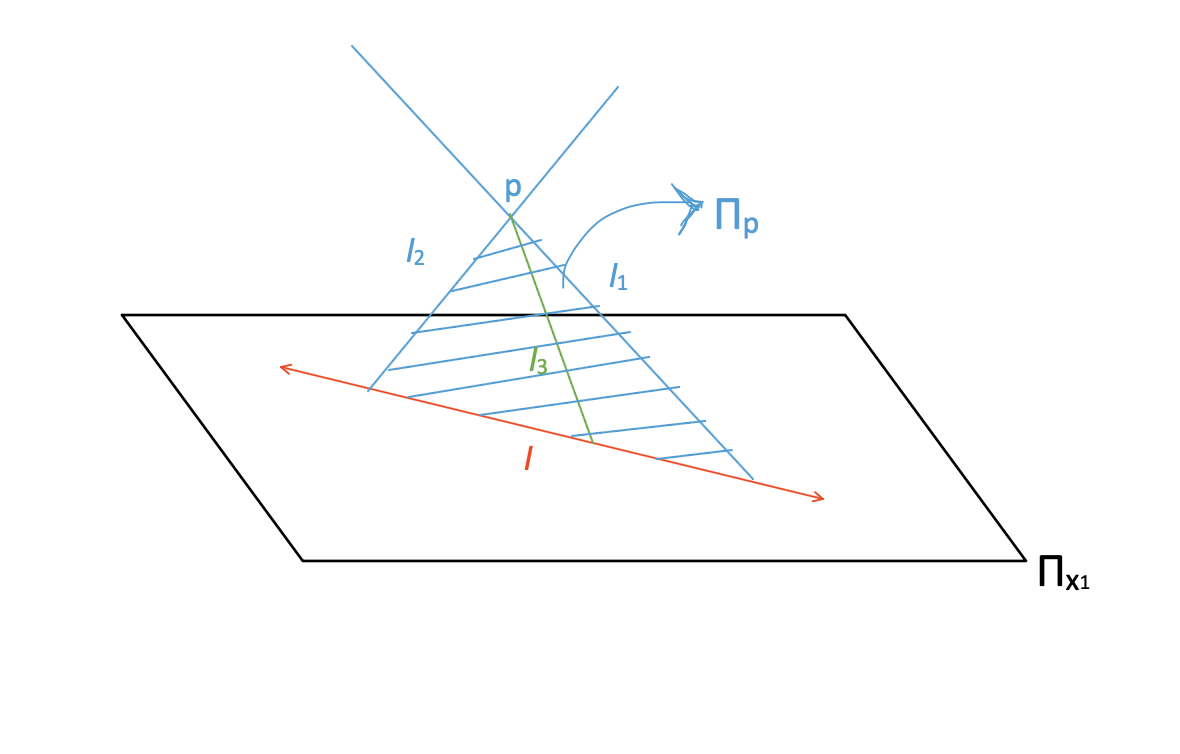}
\centering
\caption[Key step in planebrush argument]{Key step in planebrush argument}
\label{pbstep}
\end{figure}
We have reduced to a 3-dimensional configuration. Let $l_3$ be any line passing through $p$, a priori not in the planebrush. By 3-dimensional geometry, either $l_3$ intersects $\Pi_{x_1}$ at a point on $l$, or else $l_3$ is parallel to $l$ (See Figure \ref{pbstep}). Thus by definition, we have $l_3\in \mathcal{L}_1$, i.e., $l_3$ is a line in the planebrush. Since $l_3$ was an arbitrary line passing through $p$, we have now shown that $p$ does not intersect any line from $\mathcal{L}_1^c$, i.e., any line from outside the planebrush. This establishes Claim 3 as $p$ was an arbitrary point in $P_1'$.

\medskip

Our next task is to bound $|P_1'|$ from below. Define
\begin{equation}\label{constdens}
\mathcal{L}_1'=\big\{l\in \mathcal{L}_1: |P_1' \cap l|>q/200\big\}.
\end{equation}
That is, $\mathcal{L}_1'$ is the set of lines in the planebrush that have at least $0.5\%$ of their points in $P_1'$. Henceforth, we shall focus our attention on the lines in $\mathcal{L}_1'$. To be more precise, we shall be interested in the pairs $(p,l)\in P_1'\times \mathcal{L}'_{1}$, where $p\in l$. We would now like to justify why this is a reasonable thing to do.

By definition, we have 
\[
|S''|=\sum_{p\in P_1'} \sum_{l \in \mathcal{L}_1} \chi_l(p)=\sum_{l \in \mathcal{L}_1} \sum_{p\in P_1'} \chi_l(p)=\sum_{l\in \mathcal{L}_1}|l\cap P_1'|.
\] 
Since we are in Case 2, we have $|S''|\geq |S|/2=q|\mathcal{L}_1|/2$. Thus we have
$$
\sum_{l\in \mathcal{L}_1'} |l \cap P_1'|+\sum_{l\in \mathcal{L}_1\setminus \mathcal{L}_1'} |l \cap P_1'| \geq \frac{q|\mathcal{L}_1|}{2},
$$
which implies
\begin{equation}\label{massconv}
\sum_{l\in \mathcal{L}_1'} |l \cap P_1'| \geq \frac{q|\mathcal{L}_1|}{2}-\sum_{l\in \mathcal{L}_1\setminus \mathcal{L}_1'} |l \cap P_1'|\geq \frac{q|\mathcal{L}_1|}{2}-\frac{q|\mathcal{L}_1|}{100}=q|\mathcal{L}_1| \frac{49}{100}.
\end{equation}
The above inequality guarantees we have kept a substantial proportion of the point-line pairs as we restrict our attention to $\mathcal{L}_1'$. Furthermore, we know that $\sum_{l\in \mathcal{L}_1'} |l \cap P_1'| \leq q |\mathcal{L}_1'|$. Combining this with \eqref{massconv} gives us
\begin{equation}\label{linescons}
    |\mathcal{L}_1'|\geq \frac{49}{100}|\mathcal{L}_1|.
\end{equation}
Thus, we are keeping a substantial fraction of the lines of the planebrush as well. We conclude that restricting our attention to $\mathcal{L}_1$ gives us the useful property \eqref{constdens}, and also we do not lose too many point-line pairs \eqref{massconv} or too many lines \eqref{linescons}. 

Note that we can foliate $\FF_q^n$ into 3-spaces as follows:
$$
\FF_q^4=\Pi_{x_1}\ \sqcup\ \Big(\bigsqcup_\alpha (V_\alpha\setminus \Pi_{x_1})\Big),
$$
where the union is taken over all 3-spaces $V_\alpha$ containing $\Pi_{x_1}$. Every line in the planebrush must be contained in some $V_\alpha$. Correspondingly, we have a decomposition of the lines in $\mathcal{L}_1'$ as
\begin{equation}\label{linesdecomp}
    \mathcal{L}_1'=\mathcal{L}_{\Pi} \sqcup \Big(\bigsqcup_\alpha \mathcal{L}_\alpha \Big).
\end{equation}

Furthermore, the planebrush set $P_1'$ decomposes as
\begin{equation}\label{pbdecomp}
    P_1'=\Big(\bigcup_{l\in \mathcal{L}_\Pi} l\cap P_1' \Big) \sqcup \Big[\bigsqcup_\alpha \Big(\bigcup_{l\in \mathcal{L}_\alpha} (l\cap P_1') \setminus \Pi_{x_1} \Big) \Big]. 
\end{equation}

We first prove a lower bound on $|\bigcup_{l\in \mathcal{L}_\alpha} (l\cap P_1') \setminus \Pi_{x_1}|$ for each $\alpha$. Fix $\alpha$, and consider the set of lines $\mathcal{L}_\alpha$ in the 3-space $V_\alpha$. For each $l\in \mathcal{L}_\alpha$, we know that $|l\cap P_1'|\geq q/100$ and that $l$ intersects $\Pi_{x_1}$ in at most one point. It follows that $|(l\cap P_1') \setminus \Pi_{x_1}|\geq q/200$. 
By hypothesis, we have $|\mathcal{L}_\alpha|\leq 3q^2$, with at most $2q$ lines contained in a common 2-plane.
We may therefore apply Proposition \ref{hbprop0}$'$, with $(P_1'\setminus \Pi_{x_1})\cap V_\alpha$ in place of $X$. We conclude that 
\begin{equation}\label{pbtohb}
    \Big|\bigcup_{l\in \mathcal{L}_\alpha} (l\cap P_1') \setminus \Pi_{x_1}\Big|\gtrsim |\mathcal{L}_\alpha|q^{1/2}.
\end{equation}
Furthermore, by Proposition \ref{cordprop}$'$, we have 
\begin{equation}\label{cordpb}
    \Big|\bigcup_{l\in \mathcal{L}_\Pi} l\cap P_1' \Big|\gtrsim |\mathcal{L}_\Pi|q \gtrsim |\mathcal{L}_\Pi|q^{1/2}.
\end{equation}
Combining \eqref{pbdecomp}, \eqref{pbtohb} and \eqref{cordpb}, we get
\begin{equation}
    |P_1'|\gtrsim |\mathcal{L}_\Pi|q^{1/2} +\sum_\alpha |\mathcal{L}_\alpha|q^{1/2}=\Big(|\mathcal{L}_\Pi| +\sum_\alpha |\mathcal{L}_\alpha|\Big)q^{1/2}.
\end{equation}
Note that \eqref{linesdecomp} implies that $|\mathcal{L}_\Pi| +\sum_\alpha |\mathcal{L}_\alpha|=|\mathcal{L}_1'|$. Therefore we get
\begin{equation}\label{pbound}
    |P_1'|\gtrsim |\mathcal{L}_1'|q^{1/2} \gtrsim |\mathcal{L}_1|q^{1/2}\gtrsim |\mathcal{L}_1|q^{1/3}.
\end{equation}

We used \eqref{linescons} in the second inequality above. Now that we have a good lower bound on the volume of the planebrush, we can apply the inductive hypothesis \eqref{thm-pb4MainInequality} to the union of the remaining lines disjoint from the planebrush (the second term on the right side of \eqref{disjunipb}). Together with \eqref{pbound}, this yields
$$
    |X|\geq |P_1'|+\Big|\bigcup_{l\in \mathcal{L}_1^c} l\Big|\gtrsim |\mathcal{L}_1|q^{1/3}+|\mathcal{L}^c_1|q^{1/3}=|\mathcal{L}|q^{1/3}.
$$
This closes the induction and completes the proof.
\end{proof}

\section{Acknowledgements}
All three authors were supported by NSERC Discovery Grants.


\begin{thebibliography}{}


\bibitem{bes} A.~Besicovitch, \textit{Sur deux questions d'integrabilite des fonctions}, J. Soc. Phys. Math. \textbf{2} (1919), 105-123.

\bibitem{bourgainguth} J.~Bourgain, L.~Guth, \textit{Bounds on oscillatory integral operators based on multilinear estimates}, Geom. Funct. Anal. \textbf{21} (2011), 1239-1295.

\bibitem{BKT} J.~Bourgain, N.~Katz, T.~Tao, \textit{A sum-product estimate in finite fields, and applications}, Geom. Func. Anal. \textbf{14} (2004), 27-57.

\bibitem{BukTi} B.~Bukh, C.~Ting-Wei, \textit{Sharp Density Bounds on the Finite Field Kakeya Problem}, Discrete Analysis \textbf{26} (2021).

\bibitem{cordoba} A.~C\'ordoba, \textit{The {K}akeya maximal function and the spherical summation multipliers}, Am. J. Math. \textbf{99} (1977), 1--22.

\bibitem{dav} R.~Davies, \textit{Some remarks on the {K}akeya problem}, Proc. Cambridge Philos. Soc. \textbf{69} (1971), 417–421.


\bibitem{dvir} Z. Dvir, \textit{On the size of Kakeya sets in finite fields}, J. Amer. Math. Soc. \textbf{22} (2009), no. 4, 1093-1097.


\bibitem{DKSS} Z.~Dvir, S.~Kopparty, S.~Saraf, M.~Sudan, \textit{Extensions to the method of multiplicities, with applications to Kakeya sets and mergers},  2009 50th Annual IEEE Symposium on Foundations of Computer Science (2009), 181-190. 

\bibitem{guth2016} L.~Guth, \textit{Degree reduction and graininess for {K}akeya-type sets in $\mathbb{R}^3$}, Rev. Mat. Iberoam. \textbf{32} (2016) 447-494.


\bibitem{guthbook} L. Guth, \textit{Polynomial Methods in Combinatorics}, University Lecture Series (64), American Mathematical Society, 2016.

\bibitem{guthzahl} L. Guth, J. Zahl, \textit{Polynomial Wolff axioms and Kakeya-type estimates in $\RR^4$}, Proc. London Math. Soc. (3) \textbf{117} (2018), no. 1, 192–220.

\bibitem{KLT} N.~Katz, I.~\L{}aba, T.~Tao, \textit{An improved bound on the {M}inkowski dimension of {B}esicovitch sets in $\mathbb{R}^3$}, Ann. of Math. \textbf{152} (2000) 383-446.


\bibitem{katzrog} N. H. Katz, K. M. Rogers, \textit{On the polynomial Wolff axioms},
Geom. Func. Anal. \textbf{28} (2018), no. 6, 1706-1716.


\bibitem{katztao} N.~Katz, T.~Tao, \textit{Recent progress on the Kakeya conjecture}, Publ. Mat. \textbf{46} (2002), 161-179.

\bibitem{katzzahl} N. H. Katz, J. Zahl, \textit{A Kakeya maximal function estimate in four dimensions using planebrushes}, Rev. Mat. Iberoam. (1) \textbf{37} (2021), 317-359.


\bibitem{mockTao} G.~Mockenhaupt, T.~Tao, \textit{Restriction and Kakeya phenomena for finite fields}, Duke Math. J. \textbf{121} (2004), 35-74.

\bibitem{RaiCh} M.~Rai Choudhuri, \textit{An improved bound on the Hausdorff dimension of sticky Kakeya sets in $\mathbb{R}^4$}, arXiv:2410.23579 (2024).

\bibitem{rogers} K.M~Rogers, \textit{The finite field Kakeya problem}, Amer. Math. Monthly \textbf{108}(8) (2001) 756-759.

\bibitem{tao4d} T. Tao, \textit{A new bound for finite field Besicovitch sets in four dimensions}, Pacific J. Math. \textbf{222}(2) (2005), 337--363.


\bibitem{wangzahl} H.~Wang, J.~Zahl, \textit{Volume estimates for unions of convex sets, and the Kakeya set conjecture in three dimensions},    arXiv:2502.17655 (2025).

\bibitem{wolff1} T. Wolff, \textit{An improved bound for Kakeya type maximal functions}, Revista Mat. Iberoamericana. \textbf{11} (1995), 651–674.

\bibitem{wolff} T. Wolff, \textit{Recent work connected with the Kakeya problem}, Prospects in Mathematics (Princeton, NJ, 1996), 129–162, Amer. Math. Soc., Providence, RI, 1999. 

\end{thebibliography}
\end{document}